\documentclass[reqno]{amsart}
\usepackage{amsfonts}

\newtheorem{theorem}{Theorem}
\theoremstyle{plain}

\newtheorem{corollary}{Corollary}

\newtheorem{definition}{Definition}

\numberwithin{equation}{section}

\begin{document}
\title[Multipliers of Grand and Small Lebesgue Spaces]{Multipliers of Grand and Small Lebesgue Spaces}
\author{A.Turan G\"{u}rkanl\i }
\address{Istanbul Arel University Faculty of Science and Letters Department
of Mathematics and Computer Sciences}
\email{turangurkanli@arel.edu.tr}
\date{}

\begin{abstract}
Let $G$ a locally compact abelian group with Haar measure $\mu$ and let $1<p<\infty. $ In the present paper we determine necessary and sufficient conditions on $G$ for the grand Lebesgue space $ L^{p),\theta}(G)$ to be a Banach algebra under convolution.Later we characterize the multipliers of the grand Lebesgues, $L^{p,)\theta}(G)$ and the small Lebesgue spaces $L^{(P'\theta}$, where $\frac{1}{p}+\frac{1}{p'}=1$
\end{abstract}

\maketitle

\section{Notations}

Let $\Omega $ be locally compact hausdorff space and let $\left( \Omega ,\mathfrak{M},\mu \right) $ be finite measure space. 
 The grand Lebesgue space $L^{p)}\left( \Omega \right) $ was introduced by Iwaniec-Sbordone in $\left[ 15\right]. $ These authors, in their studies related with the integrability  properties of the Jacobian in a bounded open set  $\Omega $, defined the grand lebesgue space. This Banach space is defined by the norm
\begin{align*}
\left\Vert {f}\right\Vert _{p)}=\sup_{0<\varepsilon \leq p-1}\left(
\varepsilon \int_{\Omega }\left\vert f\right\vert ^{p-\varepsilon }d\mu
\right) ^{\frac{1}{p-\varepsilon }}\tag{1}
\end{align*}
where $1<p<\infty$ .  For  $0<\varepsilon \leq p-1,$

$$L^{p}\left( \Omega \right) \subset L^{p)}\left( \Omega \right) \subset \L^{p-\varepsilon }\left( \Omega \right)$$

 hold. For some properties and applications of $L^{p)}\left( \Omega \right) $ spaces we refer to papers $\left[
1\right] ,\left[ 3\right] ,\left[ 5\right] ,\left[ 6\right] ,$ $\left[ 13
\right] $and $\left[ 15\right] .$ 

 A generalization of the grand Lebesgue
spaces are the spaces $L^{p),\theta }\left( \Omega \right) ,$ $\theta \geq 0,$
defined by the norm
\begin{equation*}
\left\Vert f\right\Vert _{p),\theta ,\Omega }=\left\Vert f\right\Vert
_{p),\theta }=\sup_{0<\varepsilon \leq p-1}\varepsilon ^{\frac{\theta }{
p-\varepsilon }}\left( \int_{\Omega }\left\vert f\right\vert ^{p-\varepsilon
}d\mu \right) ^{\frac{1}{p-\varepsilon }}=\sup_{0<\varepsilon \leq
p-1}\varepsilon ^{\frac{\theta }{p-\varepsilon }}\left\Vert {f}\right\Vert
_{p-\varepsilon };\tag{2}
\end{equation*}%
when $\theta =0$ the space $L^{p),0}\left( \Omega \right) $ reduces to Lebesgue
space $L^{p}\left( \Omega \right) $ and when $\theta =1$ the space $%
L^{p),1}\left( \Omega \right) $ reduces to grand Lebesgue space $L^{p)}\left(
\Omega \right)$,  $\left( \text{see }\left[ 1\right] ,\left[ 12\right]
\right) $ . We have  for  $0<\varepsilon \leq p-1,$

$$L^{p}\left( \Omega \right) \subset L^{p),\theta }\left( \Omega \right) \subset
L^{p-\varepsilon }\left( \Omega \right) .$$ It is known that the subspace ${C_{c}^{\infty }(\Omega)}$ is not dense in $L^{p)}\left( \Omega \right). $  Its closure  $[L^p]_{p),\theta}$ consists of functions $f \in L^{p)}\left(\Omega \right)$  such that
\begin{equation}
\lim_{\varepsilon \rightarrow 0}\varepsilon ^{\frac{\theta }{p-\varepsilon }
}\left\Vert f\right\Vert _{p-\varepsilon}=0, \left[12\right]. \tag{3}
\end{equation}
 It is also known that the generalized grand Lebesgue space  $L^{p),\theta }\left( \Omega \right) ,$ is not reflexive.
Different properties and applications of these spaces were discussed in $[ 1] ,[10],[11],[12] $ and $[ 13] .$

Let $p^{^{\prime }}=\frac{p}{p-1},$ $1<p<\infty .$ First consider an
auxiliary space namely $L^{(p^{\prime },\theta }( \Omega)
,\theta \ge 0,$ defined by 
\begin{equation*}
\left\Vert g\right\Vert _{(p^{^{\prime }},\theta
}=\inf_{g=\sum\limits_{k=1}^{\infty }g_{k}}\left\{
\sum\limits_{k=1}^{\infty }\inf_{0<\varepsilon \leq p-1}\varepsilon ^{-%
\frac{\theta }{p-\varepsilon }}\left( \int_{\Omega }\left\vert
g_{k}\right\vert ^{\left( p-\varepsilon \right) ^{^{\prime }}}dx\right) ^{%
\frac{1}{\left( p-\varepsilon \right) ^{^{\prime }}}}\right\}\tag{4}
\end{equation*}%
where the functions $g_{k},k\in \mathbb{N},$ being in $\mathcal{M}_{0},$ the set of all real valued measurable
functions, finite a.e. in $\Omega .$ After this definition we define the space $L^{p)^{^{\prime }},\theta }\left( \Omega \right)$ by the norm
\begin{equation*}
L^{p)^{^{\prime }},\theta }\left( \Omega \right) =\left\{ g\in \mathcal{M}%
_{0}:\left\Vert g\right\Vert _{p)^{^{\prime }},\theta }<+\infty \right\} ,
\end{equation*}%
where%
\begin{equation*}
\left\Vert g\right\Vert _{p)^{^{\prime },\theta }}=\sup_{\substack{ 0\leq
\psi \leq \left\vert g\right\vert  \\ \psi \in L^{(p^{\prime },\theta }}}%
\left\Vert \psi \right\Vert _{(p^{^{\prime }},\theta }.\tag{5}
\end{equation*}%
For $\theta =0$ it is $\left\Vert f\right\Vert _{(p^{^{\prime
}},0}=\left\Vert f\right\Vert _{p)^{^{\prime }},\theta },\left[ 4\right] ,[12] .$ It is known by Theorem 3.3. in [4] that 

$$\left\Vert g\right\Vert _{p)^{^{\prime },\theta }}\equiv\left\Vert g\right\Vert _{(p^{^{\prime }},\theta}$$ and
 \begin{equation*}
\left\Vert g\right\Vert _{p)^{^{\prime },\theta }}=\left\Vert g\right\Vert _{(p^{^{\prime }},\theta
}=\inf_{g=\sum\limits_{k=1}^{\infty }g_{k}}\left\{
\sum\limits_{k=1}^{\infty }\inf_{0<\varepsilon \leq p-1}\varepsilon ^{-%
\frac{\theta }{p-\varepsilon }}\left( \int_{\Omega }\left\vert
g_{k}\right\vert ^{\left( p-\varepsilon \right) ^{^{\prime }}}dx\right) ^{
\frac{1}{\left( p-\varepsilon \right) ^{^{\prime }}}}\right\}.\tag{6}
\end{equation*}

Let $(B,\|.\|_{B})$  be a Banach space and let $(A,\|.\|_{A})$ be a Banach algebra . If $B$ is an algeraic $A$- module, and $\|ab\|_{B} \leq\|a\|_A\|b\|_A$ for all $a\in A, b\in B$, then B is called a Banach A-module.  If  $B_1,B_2$ are Banach $A$-modules, then the a multiplier  (or module homomorphism ) from $B_1$ to $B_2$ is a bounded linear operator $T$ from  $B_1$ to $B_2$ which commutes with the module multiplication, i.e. $T(ab)=aT(b)$   for all $a\in A, b\in B_1$, The space of multipliers from $B_1$ to $B_2$ is denoted by $M(B_1 ,B_2)$ (or $ Hom(B_1 ,B_2)$), [8],[14],[17] [18] and [19].

\section{Main Results} 
\begin{theorem}
Let $G$ be a locally compact Abelian group, and $\mu$ its Haar measure. Then the generalized grand Lebesgue space $L^{p),\theta }$, $1<p<\infty $ is a Banach algebra under convolution if and only if $G$ is compact.
\end{theorem}
\begin{proof}
Let $G$ be a compact Abelian group. For the proof of  the generalized grand Lebesgue space $L^{p),\theta }$, $1<p<\infty $ is a Banach algebra under convolution, it is enough to show that 
\begin{equation*}
\left\Vert f\ast g\right\Vert _{p),\theta }\leq \left\Vert f\right\Vert
_{p),\theta }\left\Vert g\right\Vert _{p),\theta }
\end{equation*}%
for all $f,g\in L^{p),\theta }\left( G\right) .$ Since $G$ is compact, then $%
L^{p}\left( G\right) $ is a Banach convolution algebra for all $p,$ $1\leq p\leq\infty,$ by Zelasko $\left[ 21
\right] .$Thus $\left( L^{p-\varepsilon }\left( G\right) ,\left\Vert
.\right\Vert _{p-\varepsilon }\right) $ is a Banach algebra under
convolution for all $0<\varepsilon \leq p-1.$ Then for all $f,g\in
L^{p-\varepsilon }\left( G\right) ,$ we have 
\begin{equation*}
\left\Vert f\ast g\right\Vert _{p-\varepsilon }\leq \left\Vert f\right\Vert
_{p-\varepsilon }\left\Vert g\right\Vert _{p-\varepsilon }.
\end{equation*}%
Thus 
\begin{eqnarray*}
\left\Vert f\ast g\right\Vert _{p),\theta } &=&\sup_{0<\varepsilon \leq
p-1}\varepsilon ^{\frac{\theta }{p-\varepsilon }}\left\Vert f\ast
g\right\Vert _{p-\varepsilon }\leq \sup_{0<\varepsilon \leq p-1}\varepsilon
^{\frac{\theta }{p-\varepsilon }}\left\Vert f\right\Vert _{p-\varepsilon
}.\sup_{0<\varepsilon \leq p-1}\varepsilon ^{\frac{\theta }{p-\varepsilon }
}\left\Vert g\right\Vert _{p-\varepsilon } \\
&\leq &\left\Vert f\right\Vert _{p),\theta }\left\Vert g\right\Vert
_{p),\theta }.
\end{eqnarray*}

Conversely Assume that $L^{p),\theta }( G) $ is a Banach algebra under convolution
for all $\theta \geq 0$ and, $1<p<\infty$. Since $L^{p),\theta }( G) $ reduces to $
L^{p}( G) $ when $\theta =0$, and $L^p(G)$ is a
Banach algebra under convolution, then $G$ is compact again from Zelasko 
$[ 21].$ 
\end{proof}
\begin{theorem}
Let $G$ be a compact Abelian group and $\mu$  its Haar measure. Then   $[L^p]_{p),\theta}$ is a Banach convolution module over itself.
\begin{proof}
Let  $f,g\in [L^p]_{p),\theta}.$ Then
\begin{equation*}
 \lim_{\varepsilon \rightarrow 0}\varepsilon ^{\frac{\theta }{p-\varepsilon }
}\| f\| _{p-\varepsilon}=0,\tag{8}
\end{equation*}
 and $$\lim_{\varepsilon \rightarrow 0}\varepsilon ^{\frac{\theta }{p-\varepsilon } 
}\| g\| _{p-\varepsilon}=0.$$ Thus $$\lim_{\varepsilon \rightarrow 0}\varepsilon ^{\frac{\theta }{p-\varepsilon } 
}\| f+g\| _{p-\varepsilon}\le\lim_{\varepsilon \rightarrow 0}\varepsilon ^{\frac{\theta }{p-\varepsilon } 
}\| f\| _{p-\varepsilon}+\lim_{\varepsilon \rightarrow 0}\varepsilon ^{\frac{\theta }{p-\varepsilon } 
}\| g\| _{p-\varepsilon}$$ and  so $f+g\in [L^p]_{p),\theta}. $ Since $G$ is compact, by Theorem 1, $f*g\in L^{p),\theta}, $ and $$\|f*g\|_{p-\varepsilon}\le\|f\|_{p-\varepsilon}\|g\|_{p-\varepsilon}.$$
 According to this inequality $$\lim_{\varepsilon \rightarrow 0}\varepsilon ^{\frac{\theta }{p-\varepsilon } 
}\| f*g\| _{p-\varepsilon}\le\lim_{\varepsilon \rightarrow 0}\varepsilon ^{\frac{\theta }{p-\varepsilon } 
}\| f\| _{p-\varepsilon}\| g\| _{p-\varepsilon}=0$$ for all $0<\varepsilon\le p-1.$ Thus we have $f*g\in[L^p]_{p),\theta}$.
We know that $[L^p]_{p),\theta}$ is closed in $L^{p),\theta}$. Then it is a Banach space. Finally  $[L^p]_{p),\theta}$ is a Banach module over itself
\end{proof} 
\end{theorem}

\begin{theorem} Let $G$ be a compact Abelian group and $\mu$ its Haar measure. Then
 
a.  The space $[L^p]_{p),\theta} $ admits the module factorization property. 

b.  We have $$M([L^p]_{p),\theta},L^{p)',\theta})=L^{p)',\theta},$$ where $M([L^p]_{p),\theta},L^{p)',\theta})$ denotes the space of multipliers from $[L^p]_{p),\theta}$ to $L^{p)',\theta}.$
\end{theorem}
\begin{proof}
a. It is known by Theorem 6 in [13] that The space $[L^p]_{p),\theta} $ admits an approximate identity bounded in
 $L^1 (G).$ We have proved in Theorem 1, that $ [L^p]_{p),\theta}$ is a Banach convolution algebra. This implies 
$[L^p]_{p),\theta}$ is a Banach convolution module over itself. Thus $$[L^p]_{p),\theta} *[L^p]_{p),\theta}\subset\ [L^p]_{p),\theta}.$$ We are going to show that $[L^p]_{p),\theta} *[L^p]_{p),\theta}$ is dense in$[L^p]_{p),\theta}$. Let $(e_\alpha)_{\alpha \in I}$ be an approximate identity in $[L^p]_{p),\theta}$ bounded in $L^1(G)$. Take any function $f\in [L^p]_{p),\theta}.$ Since $f*e_\alpha\in[L^p]_{p),\theta} *[L^p]_{p),\theta}$, for all $\alpha\in I,$ we have $(f*e_{\alpha})_{\alpha\in I}\subset[L^p]_{p),\theta} *[L^p]_{p),\theta}.$ Since $e_{\alpha})_{\alpha\in I}$ is an approximate identity, then  the sequence $(f*e_{\alpha})_{\alpha\in I}$ converges to $f.$ Hence $[L^p]_{p),\theta} *[L^p]_{p),\theta}$ is dense in $[L^p]_{p),\theta}$. So, by the module factorization theorem (see 8.5 Theorem in [20]), we write  $$[L^p]_{p),\theta} *[L^p]_{p),\theta}= [L^p]_{p),\theta}.$$ This completes the proof of first part.

b. It is known by Corollary 3.8, in [11] that the dual of $[L^p]_{p),\theta} $ is isometrically isomorphic to $L^{p)',\theta}$, i.e $$([L^p]_{p),\theta})^*=L^{p)'\theta}$$. Then by (a) and  by Corollary 2.13 in [19] we write $$Hom_{L^1}([L^p]_{p),\theta},L^{p)',\theta})=([L^p]_{p),\theta}*[L^p]_{p),\theta})^*=([L^p]_{p),\theta})^*=L^{p)'\theta}.$$ Theorem 3 is therefore proved.
\end{proof}
\begin{corollary}
Let $G$ be compact Abelian group.For every $T\in M(L^{p),\theta},L^{(p',\theta}),$ there exists a function $f\in L^{(p',\theta}(G) $ such that $f$ is the restriction of $T$, to $[L^p]_{p)\theta}$, i.e. $T/[L^p]_{p)\theta}=f$.
\end{corollary}
\begin{proof}
It is easy to show that $T/[L^p]_{p)\theta}$ is linear and continuous. Let $g,h\in [L^p]_{p)\theta}$.By Theorem 3, $g*h\in[L^p]_{p)\theta} $.Then $$T/[L^p]_{p)\theta}(g*h)=T(g*h)=g*Th=g*T/[L^p]_{p)\theta}.$$ That means,$T/[L^p]_{p)\theta}$ commutes with the module multiplication. If we set $T/[L^p]_{p)\theta} =f,$ by Theorem 3,$f\in M([L^p]_{p),\theta},L^{p)',\theta})=L^{p)',\theta}.$ 
\end{proof}
\begin{theorem}
Let $G$ be a locally compact Abelian group with Haar measure $\mu$ and $\mu(G) $ is finite. If $1<p<\infty$, then the small Lebesgue space $L^{(p'}(G)$ is a Banach convolution module over $L^1(G)$, where $\frac{1}{p}+\frac{1}{p'}=1.$
\end{theorem}
\begin{proof}
 Let $f\in L^1(G)$,  $  g \in L^{(p',\theta}$. If  $g=\sum\limits_{k=1}^{\infty }g_{k}$, it is easy to show that  $f*g=\sum\limits_{k=1}^{\infty }f*g_{k}.$  Since $L^{(p-\varepsilon)^{'}}(G)$ is Banach convolution module over $L^1(G)$, we have $$\|f*g_{k}\|_{(p-\varepsilon)^{'}}\leq\|f\|_1\|g_{k}\|_{(p-\varepsilon)'}$$ for all $k$, this implies  
\begin{equation*}
\left\Vert f*g\right\Vert _{p)^{^{\prime },\theta }}=\left\Vert f*g\right\Vert _{(p^{^{\prime }},\theta}\leq
\inf_{g=\sum\limits_{k=1}^{\infty }g_{k}}\left\{
\sum\limits_{k=1}^{\infty }\inf_{0<\varepsilon \leq p-1}\varepsilon ^{-%
\frac{\theta }{p-\varepsilon }}\left( \int_{\Omega }\left\vert
f*g_{k}\right\vert ^{\left( p-\varepsilon \right) ^{^{\prime }}}dx\right) ^{
\frac{1}{\left( p-\varepsilon \right) ^{^{\prime }}}}\right\} 
\end{equation*}
$$=\inf_{g=\sum\limits_{k=1}^{\infty }g_{k}}\left\{\sum\limits_{k=1}^{\infty }\inf_{0<\varepsilon \leq p-1}\varepsilon ^{-%
\frac{\theta }{p-\varepsilon }}\|f*g\|_{(p-\varepsilon)'}\right\}$$
$$\leq\inf_{g=\sum\limits_{k=1}^{\infty }g_{k}}\left\{\sum\limits_{k=1}^{\infty }\inf_{0<\varepsilon \leq p-1}\varepsilon ^{-%
\frac{\theta }{p-\varepsilon }}\|f\|_1\|g\|_{(p-\varepsilon)'}\right\}$$
$$=\|f\|_1\inf_{g=\sum\limits_{k=1}^{\infty }g_{k}}\left\{\sum\limits_{k=1}^{\infty }\inf_{0<\varepsilon\leq p-1}\varepsilon^{-\frac{\theta}{p-\varepsilon}}\|g\|_{(p-\varepsilon)'} \right\}=\|f\|_1\|g\|_{(p',\theta}
$$
This completes the proof.
\end{proof}
\begin{theorem}
Let $G$ be a locally compact Abelian group with Haar measure $\mu$ and $\mu(G) $ is finite. Then the space $L^{(P',\theta}(G)$ admits an approximate identity bounded in $L^1(G)$.
\end{theorem}
\begin{proof}
It is known that $$ L^{p'+\varepsilon}\subset L^{(p',\theta}\subset L^{p'}$$ for all $\epsilon>0.$ It is also known that $ L^{P'+\varepsilon}$ has an approximate identity $(e_\alpha)_{\alpha\in I }$ bounded in  $L^1(G)$. Assume that $\|e_\alpha\|_1\le C$ for all $\alpha\in I$ and $C>1.$ Since $C^\infty_{c}(G)$ is dense in $L^{(p',\theta}$, for given any $f \in L^{(P',\theta}$ and $\eta>0,$ there exists $g\in C^\infty_{c}(G)$ such that $$\|f-g\|_{p'+\varepsilon} \le\frac{\eta}{3C} $$. Since $ C^\infty_{c}(G)\subset L^{p'+\varepsilon}$ and so $g\in L^{p'+\varepsilon}$, for this given $\eta>0$ there exists $\alpha_0\in I$ such that for all $\alpha>\alpha_{0}$, $$\|e_{\alpha}*g-g\|_{(p',\theta}\le\|e_{\alpha}*g-g\|_{p'+\varepsilon}\le\frac{\eta}{3}.$$ By Theorem 4, for all $\alpha\in I$  we have $f*e_{\alpha}\in L^{(p'}and $$$\|f-e_{\alpha*f\|_{(p',\theta}}=\|f-e_{\alpha}*f+e_{\alpha}*g-e_{\alpha}*g+g-g\|_{(p',\theta}$$$$\le\|f-g\|_{(p',\theta}+\|g-e_{\alpha}*g\|_{(p',\theta}+\|e_{\alpha}*f-e_{\alpha}*g\|_{(p',\theta}$$$$=\|f-g\|_{(p',\theta}+\|g-e_{\alpha}*g\|_{(p',\theta}+\|e_{\alpha}*(f-g)\|_{(p',\theta}\le\frac{\eta}{3C}+\frac{\eta}{3}+$$ $$+\|e_\alpha\|_1\|f-g\|_{(p',\theta}\le\frac{\eta}{3C}+\frac{\eta}{3}+\frac{\eta}{3C}C\le\eta.$$ This completes the proof. 
\end{proof}

From Theorem 4, and theorem 5, we give the following Theorem.

\begin {theorem} 
 Let $G$ be a locally compact Abelian group and $\mu$ its Haar measure. Then $M(L^1(G),L^{(p',\theta}(G))=L^{p),\theta}(G)$ 
\end {theorem}                             
\begin{proof}
Let  $(e_\alpha)_{\alpha\in I }$ be an approximate identity of   the space $L^{(P',\theta}(G)$ bounded in  $L^1(G)$ and let $f\in  L^{(P',\theta}(G).$ Then from Theorem 4 and Theorem 5, $e_\alpha*f\in L^{(P',\theta}(G)$ for all $\alpha\in I,$ and the net $(e_\alpha*f)_{\alpha\in I}$ converges to $f$ in $ L^{(P',\theta}(G)$. Hence  $L^1(G)*L^{(P',\theta}(G)$ is dense in $ L^{(P',\theta}(G).$ Then by the module factorization theorem  $L^1(G)*L^{(P',\theta}(G)=L^{(P',\theta}(G)$. Since $$ (L^{(p'\theta}(\Omega))'= (L^{(p'\theta}(\Omega))^{*}=L^{p),\theta},$$ then we have $$M(L^1,L^{(p'\theta})=(L^1(G)*L^{(P',\theta}(G))^{*}=L^{p),\theta}.$$
\end{proof}
Now we will extend this result from $L^1$ to sub spaces of small Lebesgue space $L^{(p'\theta}(\Omega)$. First we need the following definition.
\begin{definition}
Let $1<p<\infty$ and let $A$ be a linear subspace of small Lebesgue space $L^{(p,\theta}(\Omega)$ with the following properties:

(i) There is a norm $ \|.\|_A$ for A such that $\|.\|\le\|.\|_{(p,\theta}$ and  $(A, \|.\|_A)$ is a Banach $L^1$-module with respect to convolution.

(ii)  There is abounded approximate identity $(e_{\alpha})_{\alpha\in \ I}$ in $L^1$ and a number $M>0$ such that $\|e_{\alpha}\|\ _1<M$ for all $\alpha\in \ I$ and $$\|f*e_{\alpha}-f\|_A\to0$$ for each $f\in\ I$.

The relative completion of A  is the space $$A^{\sim}=\{f\in\ L^{(P,\theta}:f*e_{\alpha}\in A ,  \alpha\in\ I, and\|f\|_{A^\sim}<\infty\}, $$ where $$\|f\|_{A^\sim}=sup_{\alpha}\|f*e_{\alpha}\|_{A}$$

\end{definition}

The proofs of the following two theorems are mutadis and mutandis same as the proofs of Lemma 1.2, Theorem 1.3 in [17] and Proposition 2.2 ,Theorem 2.6 in [9].

\begin {theorem}
Let A be as in  Definition 1.Then we have 

(i) If $f\in A^\sim$ and $g\in L^1$, then $f*g\in A^\sim$ and $$\|f*g\|_{A^\sim}\le\|f\|_{A^\sim}\|g\|_1.$$

(ii) The norms $ \|.\|_{A^\sim}$ and $ \|.\|_A.$ are equivalent on A.

(iii) A is closed subspace of $A^\sim.$
\end {theorem}
\begin {theorem}
Let A be as in Definition 1 and let $1<p<\infty$. Then the space of multipliers from  $L^1(G)$ to $A$,(i.e  $M(L^1(G),A)$) and $A^\sim$ are algebrically isomorphic and homeomorphic (i.e .$M(L^1(G),A) \cong A^\sim$).
\end {theorem}


\begin{thebibliography}{99}

\bibitem{Ana1} Anatriello G, Chil R and Fiorenza A. Identification of fully
Measurable Grand Lebesgue Spaces. Journal of Function Spaces, Vol. 2017.

\bibitem{Ana2} Anatriello G. Iterated grand and small Lebesgue spaces. Collect.
Math. 2014; 64: 273--284.

\bibitem{} Capone C, Formica MR, Giova R. Grand Lebesgue spaces with respect
to measurable functions. Nonlinear Analysis 2013; 85:  125--131.

\bibitem{} Capone C, and Fiorenza A. On small Lebesgue spaces. Journal of
function spaces and applications  2005; 3 (1t): 73--89.

\bibitem{} Castillo R. E, Raferio H. Inequalities with conjugate exponents
in grand Lebesgue spaces. Hacettepe Journal of Mathematics and Statistics 
2015; 44 (1) :  33-39.

\bibitem{} Castillo R. E, Raferio H. An Introductory Course in Lebesgue
Spaces. Springer International Publishing Switzerland  2016.

\bibitem{DK} Danelia N.,  Kokilashvili V. On the approximation of periodic
functions within the frame of grand Lebesgue spaces. Bulletin of the
Georgian national academy of sciences  2012; 6(2): 11--16.

\bibitem {} Dogan M., Gurkanli AT. Multipliers of the space $S_w(G).$ Mathematica Balcanica $2001;$ new series 15:3-4.

\bibitem{DG} Duyar C., Gurkanli AT. Multipliers and relative completion in $L_{\omega}^p$. Turk J Math 2007;31:181-191.

\bibitem{Fio1} Fiorenza A, and Karadzhov GE. Grand and small Lebesgue spaces and
their analogs, Journal for Analysis and its Applications 2004; 23 (4) :  657--681.

\bibitem{Fio2} Fiorenza A. Duality and reflexity in grand Lebesgue spaces,
Collect. Math. 2000; 51  ( 2) :  131--148.

\bibitem{} Greco L, Iwaniec T, Sbordone C. Inverting the p-harmonic
operator, Manuscripta Math.1997;92:259--272. 

\bibitem{Gur1} Gurkanli AT. Inclusions and the approximate identities of the
 generalized grand Lebesgue spaces, Turk J Math.2018 ;42:3195-3203.

\bibitem{Gur2} Gurkanli AT. Multipliers of some Banach ideals and Wiener-Ditkin sets, Mathematica Slovaca 2005; 55: (2)237-248.

\bibitem{} Iwaniec T, Sbordone C. On the integrability of the Jacobian under
minimal hypotheses, Arc. Rational Mech. Anal. 1992;119:129--143.
 
\bibitem{}Larsen R. Banach Algebras an Introduction.New York, NY, USA: Marcel Dekker Inc., 1973.

\bibitem{} Oztop S, Gurkanli AT. Multipliers and tensor products of weighted $L^p$-spaces. Acta Math Sci 2001;21B: 41-49.

\bibitem{} Quek TS, and Yap YH. Multipliers from $L^1(G)$  to Lipschitz Space, Journal of  Mathematical Analysis and Applications. 1979;69:531-539

\bibitem{}Rieffel MA. Induced Banach representation of Banach algebras and locally compact groups. J. Funct Anal 1967;1:443-491.

\bibitem{}Wang HC. Homogeneous Banach Algebras. New York, NY, USA: Marcel Dekker Inc., 1977.

\bibitem{} Zelazko W. On the algebras $L^{p}$ of locally compact groups,
Colloquium Mathematicum, Vol. VIII, Fasc. 1,1961,115-120.

\end{thebibliography}
\end{document}